\documentclass[12pt, reqno]{amsart}
\usepackage{amsmath, amsthm, amscd, amsfonts, amssymb, graphicx, color}
\usepackage[bookmarksnumbered, colorlinks, plainpages]{hyperref}
\hypersetup{colorlinks=true,linkcolor=red, anchorcolor=green, citecolor=cyan, urlcolor=red, filecolor=magenta, pdftoolbar=true}

\newtheorem{theorem}{Theorem}[section]

\newtheorem{proposition}[theorem]{Proposition}
\newtheorem{corollary}[theorem]{Corollary}
\theoremstyle{definition}
\newtheorem{definition}[theorem]{Definition}
\newtheorem{example}[theorem]{Example}

\theoremstyle{remark}

\numberwithin{equation}{section}

\begin{document}
\setcounter{page}{1}

\title[ Character cyclic amenability]{ Character cyclic amenability of Banach algebras}

\author[ H. Sadeghi]{ Hamid Sadeghi Nahrekhalaji}

\address{ Department of Mathematics, Fari.C., Islamic Azad University, Isfahan, Iran.}
\email{\textcolor[rgb]{0.00,0.00,0.84}{ h.sadeghi9762@gmail.com, sadeghi.9762@iau.ac.ir}}


\begin{abstract}
 In this paper, we introduce and study the notion of character cyclic amenability for Banach algebras, which combines the cyclicity condition on derivations with the character-dependent module structure arising in character amenability. We establish several basic properties of this notion and investigate its behavior under standard Banach algebra constructions. In particular, we study character cyclic amenability for opposite algebras, projective tensor products, $\ell^1$-direct sums, quotient algebras, dense homomorphic images, and $\theta$-Lau products. We also provide examples illustrating the relationship between character cyclic amenability and other amenability-type properties, including examples showing that character cyclic amenability does not imply cyclic amenability. Finally, we characterize character cyclic amenability for module extension Banach algebras in terms of $\phi$-cyclic derivations and derivations associated with the underlying Banach module.
\end{abstract} \maketitle

\section{\bf Introduction}

Amenability and its various generalisations play a central role in the theory of Banach algebras. Since Johnson's characterization of amenable Banach algebras in terms of continuous derivations into dual Banach modules, homological methods have provided a powerful framework for studying the algebraic and analytic properties of Banach algebras \cite{9}. Over the past decades, several weaker forms of amenability have been introduced to extend this theory to Banach algebras that are not amenable in the classical sense. Among these, weak amenability and cyclic amenability have received considerable attention. Weak amenability requires every continuous derivation from a Banach algebra into its dual to be inner  \cite{31}, whereas cyclic amenability concerns derivations satisfying an additional cyclicity condition \cite{320, 32}. These notions have been studied for various classes of Banach algebras and have led to a number of results concerning standard constructions such as tensor products, direct sums, module extensions, and Lau products \cite{100, 101}.

A complementary direction in the development of amenability theory is provided by character-based versions of homological properties. In this setting, one of the module actions is determined by a character of the Banach algebra, leading to notions such as character amenability, character contractibility, character biflatness, character biprojectivity, and weak character amenability. The notions of $\varphi$-amenability and $\varphi$-contractibility were introduced and studied in \cite{105} and \cite{104}, respectively. Character amenability was introduced by Monfared in \cite{106} and has subsequently been investigated from various perspectives, including weak character amenability, essential character amenability, character contractibility, and related homological properties \cite{110, 111, 113}. These character-based notions provide a framework for studying homological properties of Banach algebras relative to individual characters; see also \cite{108, 1111, 112, 120}.

Although cyclic amenability and character amenability have been extensively studied, their combination leads to a natural class of amenability-type properties that has not been investigated in the same framework. In particular, the cyclicity condition on derivations and the character-dependent module structure suggest considering derivations whose target module is determined by a character while simultaneously satisfying the cyclicity condition. This motivates the introduction of character cyclic amenability. More precisely, for a Banach algebra $A$ and $\phi\in\Delta(A)\cup\{0\}$, we consider continuous derivations from $A$ into $(_{\phi} A)^*$ satisfying the cyclicity condition. The resulting notion combines two distinct aspects of amenability theory and allows us to examine how cyclicity interacts with character-dependent module structures.

Motivated by the above considerations, we introduce the notion of character cyclic amenability. Let $A$ be a Banach algebra and let $\phi\in\Delta(A)\cup\{0\}$. We consider $A$ as a Banach $A$-bimodule, denoted by $_{\phi}A$, with the module actions given by
$$a\cdot x=\phi(a)x,\qquad x\cdot a=xa$$
for $a\in A$ and $x\in A$. A continuous derivation $D:A\longrightarrow (_{\phi} A)^*$ is called $\phi$-cyclic if
$$\langle D(a),b\rangle+\langle D(b),a\rangle=0$$
for all $a,b\in A$. We say that $A$ is character cyclic amenable if, for every $\phi\in\Delta(A)\cup\{0\}$, every continuous $\phi$-cyclic derivation from $A$ into $(_{\phi} A)^*$ is inner. Thus, character cyclic amenability combines the character-dependent module structure with the cyclicity condition on derivations.

The main results of this paper can be summarized as follows. We first establish several basic properties of character cyclic amenability and investigate its behavior under standard constructions of Banach algebras. In particular, we study opposite algebras, projective tensor products, $\ell^1$-direct sums, quotient algebras, and dense homomorphic images. We then investigate character cyclic amenability for $\theta$-Lau products and obtain corresponding permanence results. Finally, we obtain a characterization of character cyclic amenability for module extension Banach algebras.  Several examples are given to illustrate the obtained results and to distinguish character cyclic amenability from related amenability properties.

\section{\bf Preliminaries}
Let $A$ be a Banach algebra. The Banach $A$-bimodule $X$ is a Banach space together with the structure of an $A$-bimodule, and there exists a constant $k>0$ such that
$$
\|a.x\|\leq k\|a\|\|x\| \;\;\;\;\;\text{and}\;\;\;\;\; \|x.a\|\leq k\|a\|\|x\|\;\;\;\;\;\;\;\;\;\; (a\in A,\; x\in X).
$$
By renorming, we can suppose that $k=1$.
Let now $X^*$ denote the topological dual of $X$.
It is known that $X^*$ is also a Banach $A$-bimodule through
$$
a\cdot f(x)=f(x\cdot a)\;\;\;\;\;\text{and}\;\;\;\;\;f\cdot a(x)=f(a\cdot x)
\;\;\;\;\;\;\;\;\;\; (x\in X,\;a\in A, \;f\in X^*).
$$

Let $A$ be a Banach algebra and let $X$ be a Banach $A$-bimodule.
The linear map $D:A\rightarrow X$ 
is called a derivation if 
$$D(ab)=D(a)\cdot b+a\cdot D(b)\;\;\;\;\;\;\;\;\;\; (a, b\in A).$$
In particular, for each $x\in X$, $\text{ad}_x$ defined by
$$
\text{ad}_x(a)=a\cdot x-x\cdot a\;\;\;\;\;\;\;\;\;\; (a\in A)
$$
is a derivation. Such derivations are called inner.
Indeed, a Banach algebra $A$ is called amenable if for any Banach $A$-bimodule $X$, every continuous derivation $D$ from $A$ into $X^*$ is inner and weakly amenable if every continuous derivation $D:A\rightarrow A^*$ is inner. A bounded derivation $D:A\rightarrow A^*$ is called cyclic if
$$\left\langle D(a),b\right\rangle+\left\langle D(b),a\right\rangle=0, $$ 
for all $a, b \in A$. A  Banach algebra $A$ is said to be cyclic amenable if every cyclic bounded derivation $D:A\rightarrow A^*$ is inner.

Given a Banach algebra $\mathcal{A}$, denote by $\Delta(\mathcal{A})$ the spectrum of $\mathcal{A}$
consisting of all nonzero characters on $\mathcal{A}$. Let $\varphi\in \Delta(\mathcal{A})$.
The notions of $\varphi$-amenability and $\varphi$-contractibility were introduced and studied by Kaniuth et al. \cite{105} and Hu et al. \cite{104}, respectively. 
Indeed, $\mathcal{A}$ is called $\varphi$-amenable (resp. $\varphi$-contractible) if
for every Banach $\mathcal{A}$-bimodule $X$ with the left module action 
$$a\cdot x=\varphi(a)x\;\;\;\;\;\;\;\;\;\;(a\in \mathcal{A},\; x\in X),$$
 every continuous derivation from $\mathcal{A}$ into $X^*$ (resp. $X$) is inner. 
 Moreover, $\mathcal{A}$ is called $0$-amenable (resp. $0$-contractible)
 if for every Banach $\mathcal{A}$-bimodule $X$, with the zero left action,
 every continuous derivation from $\mathcal{A}$ into $X^*$ (resp. $X$) is inner.
 We also recall from \cite{104} that $\mathcal{A}$ is called character contractible if it is $\varphi$-contractible for all $\varphi\in\Delta(\mathcal{A})\cup\{0\}$.
 Furthermore, due to Monfared \cite{106}, $\mathcal{A}$ is called character amenable if it has a bounded right approximate identity and it is $\varphi$-amenable for all $\varphi\in\Delta(\mathcal{A})$.

Let $X$ be a Banach right $A$-module and $\phi\in\Delta(A)\cup\{0\}$. We
denote by $_{\phi}X$ the Banach $A$-bimodule $X$ such that the left action satisfies $a\cdot x=\phi(a)x~(a\in A, x\in X)$.  Then $(_{\phi}X)^*$ is also a Banach $A$-bimodule with the natural dual actions.
We always consider $A$ as a Banach right $A$-module with the actions obtained from the multiplication of $A$. Then
 $_\phi A$  is a Banach $A$-bimodule with the actions defined by 
\[ a\cdot x=\phi(a)x,~~~x\cdot a=xa~~~(a,x\in A).\]
Moreover, $(_\phi A)^*$ is a Banach $A$-bimodule with the left and right actions defined by
\[ a\cdot f=af, ~ f\cdot a=\phi(a)f,\] 
 for all $a\in A$ and $f\in A^*$.
Here $af\in A^*$ is defined by $(af)(b)=f(ba)$ for all $b\in A$ (see \cite{110}).
 A Banach algebra $A$ is called weakly $\phi$-amenable if  every derivation $D:A\longrightarrow (_\phi A)^*$ is inner. Also, $A$ is called weakly character amenable whenever $A$ is weakly $\phi$-amenable for all $\phi\in\Delta(A)\cup\{0\}$.



\section{ $\phi$-cyclic amenability }
We commence this section with the following definition:

   \begin{definition}
 Let $A$ be a Banach algebra and let $\phi\in\Delta(A)\cup\{0\}$. A continuous derivation
$D:A\longrightarrow(_{\phi} A)^*$ is called a $\phi$-cyclic derivation if
$$\langle D(a),b\rangle+\langle D(b),a\rangle=0$$
for all $a,b\in A$. We say that $A$ is $\phi$-cyclic amenable if every $\phi$-cyclic derivation from $A$ into $(_{\phi} A)^*$ is inner. Moreover, $A$ is called character cyclic amenable if it is $\phi$-cyclic amenable for every $\phi\in\Delta(A)\cup\{0\}$.
\end{definition}

\begin{theorem}\label{the 2.3}
Let $A$ be a Banach algebra. If $A$ has a bounded left approximate identity,
then $A$ is  character cyclic amenable. In particular, the unitization $A^\sharp$ of $A$ is 
character cyclic amenable.
\end{theorem}
\begin{proof}
Suppose that $A$ has a bounded left approximate identity. By Theorem 2.2 of \cite{110}, $A$ is weakly character amenable. On the other hand, every weakly character amenable Banach algebra is character cyclic amenable. Hence, $A$ is character cyclic amenable.

Since the unitization $A^\sharp$ is a unital Banach algebra, it has a bounded left identity and, consequently, a bounded left approximate identity. Therefore, by the first part, $A^\sharp$ is character cyclic amenable. 
\end{proof}
It is known that the Banach algebra $A=\ell^2(\mathbb{N})$ is not weakly $0$-amenable (see Example 2.3 of \cite{110}). However, as we show in the following example, $A$ is $0$-cyclic amenable and, indeed, character cyclic amenable.

\begin{example}
Let $A = \ell^2(\mathbb{N})$ be endowed with the pointwise multiplication  \[
(ab)(n)=a(n)b(n),\qquad a,b\in A,\ n\in\mathbb{N}.
\]
For each $n\in\mathbb{N}$, let $e_n$ denote the sequence defined by
$$
e_n(k)=
\begin{cases}
1,& k=n,\\
0,& k\neq n.
\end{cases}
$$
Then
$$
e_ne_m=
\begin{cases}
e_n,& n=m,\\
0,& n\neq m.
\end{cases}
$$
The sequence $u_N=\sum_{n=1}^{N}e_n,~ N\in\mathbb{N},$ is a left approximate identity for $A$. Indeed, for every $a\in A$, $$||u_Na-a||_2^2=\sum_{n>N}|a(n)|^2\longrightarrow 0.$$
However,
$||u_N||_2=\sqrt{N}\longrightarrow\infty,$ and hence $A$ has no bounded left approximate identity.
We claim that $A$ is character cyclic amenable. First, let $\phi=0$, and let
$ D:A\longrightarrow(_0A^*)$ be a $0$-cyclic derivation. Put $D(e_n)=f_n$. Since $e_n^2=e_n$, the derivation identity gives
$f_n=e_n\cdot f_n.$
Consequently, $f_n=c_ne_n^*$ for some $c_n\in\mathbb{C}$. Since $D$ is $0$-cyclic, we have
$$0=\langle e_n,D(e_n)\rangle+
\langle e_n,D(e_n)\rangle
=2c_n.
$$
Thus $c_n=0$, and hence $D(e_n)=0~(n\in\mathbb{N}).$  Since $\operatorname{span}\{e_n:n\in\mathbb{N}\}$ is dense in $A$ and $D$ is continuous, it follows that $D=0$. Therefore every $0$-cyclic derivation is inner.

Now let $\phi=\phi_s\in\Delta(A)$, where $\phi_s(a)=a(s)$ for some $s\in\mathbb{N}$. Let
$D:A\longrightarrow(_{\phi_s}A)^*$
be a  $\phi_s$-cyclic derivation. Write $D(e_n)=f_n.$ For $n\neq s$, since $e_n^2=e_n$, we obtain
$f_n=e_n\cdot f_n,$ and hence $f_n=c_ne_n^*$ for some $c_n\in\mathbb{C}$.
On the other hand, since $e_ne_s=0$ for $n\neq s$, the derivation identity yields
$$0=e_n\cdot D(e_s)+D(e_n)\cdot e_s
=e_n\cdot f_s+f_n. $$
Thus $f_n=-f_s(n)e_n^*.$ Since $D$ is $\phi_s$-cyclic,
$$0=\langle e_n,D(e_s)\rangle+
\langle e_s,D(e_n)\rangle
=f_s(n)$$ for every $n\neq s$. Hence $f_s(n)=0~(n\neq s). $
Moreover, applying the cyclic condition to $e_s$ and itself gives $$0=2\langle e_s,D(e_s)\rangle=2f_s(s),$$
so $f_s(s)=0$. Therefore $f_s=0$, and consequently $f_n=0$ for every $n\neq s$. It follows that
$D(e_n)=0~(n\in\mathbb{N}).$ By density and continuity, $D=0$. Thus every $\phi_s$-cyclic derivation is inner.
Since $\Delta(A)=\{\phi_s:s\in\mathbb{N}\},$ we conclude that $A=\ell^2(\mathbb{N})$ is character cyclic amenable,
although it has an unbounded approximate identity.
\end{example}

The following example presents a Banach algebra that is character cyclically amenable but not cyclically amenable.
\begin{example}
Let $A$ be the two-dimensional Banach algebra $\Bbb C^2$  equipped with the norm $||(x,y)|| = |x| + |y|$ and the zero multiplication $(x,y)(u,v)=0~(x,y), (u,v)\in A.$ Consider the unitization $B = A^\sharp $. The multiplication on $B$ is defined by $$(x,\lambda)(y,\mu)=(\lambda y+\mu x,\lambda \mu)~(x,y\in A, \lambda, \mu\in \Bbb C).$$ We equip $B$ with the norm $||(x,\lambda)||=||x||+|\lambda|.$ We now prove that $B$ is not cyclically amenable. Define a skew-symmetric bilinear form $\omega$ on $A$ by  $$\omega((x_1,x_2),(y_1,y_2))=x_1y_2-x_2y_1~((x_1,x_2),(y_1,y_2)\in A).$$ Define a continuous linear mapping  $D:B\longrightarrow B^*$ by 
$$\left\langle D(x,\lambda),(y,\mu)\right\rangle=\omega(x,y)~(x,\lambda), (y,\mu)\in B.$$ Clearly, $D$ is nonzero.  Since $\omega$ is skew-symmetric, we have $\omega(y,x) =-\omega(x,y)$. It follows that  $D$ is a derivation. Moreover, $D$ is cyclic. Indeed, for arbitrary $(x,\lambda)$ and $(y,\mu)\in B$, 
$$\left\langle D(x,\lambda),(y,\mu)\right\rangle+\left\langle D(y,\mu),(x,\lambda)\right\rangle=\omega(x,y)+\omega(y,x)=0.$$ Hence $D$ is a cyclic derivation.  Since $B$ is commutative, it follows that, for every $f$ in $B^*$, the corresponding inner derivation is identically zero; that is  $\delta_f(a) = a\cdot f-f\cdot a = 0.$ The fact that $D$ is nonzero implies that $D$ cannot be inner. Thus, $B$ admits a non-inner cyclic derivation from $B$ into $B^*$. Consequently, $B$ is not cyclically amenable. Moreover, $B$ is a unital Banach algebra and hence has a bounded left approximate identity. Therefore, by Theorem \ref{the 2.3},  $B$ is  character cyclic amenable.
\end{example}

\begin{example}\label{exm 3.5}
Let $A=\mathbb{C}^{2}$ be equipped with the norm $||(x_{1},x_{2})||=|x_{1}|+|x_{2}|$
and the zero multiplication $$(x_{1},x_{2})(y_{1},y_{2})=(0,0),
\qquad (x_{1},x_{2}),(y_{1},y_{2})\in A.$$
Then $A$ is a Banach algebra and $\Delta(A)=\varnothing$. Define a continuous linear mapping
$D: A\longrightarrow (_0A)^* $ by
$$
\langle D(x_{1},x_{2}),(y_{1},y_{2})\rangle
=x_{1}y_{2}-x_{2}y_{1},
\qquad (x_{1},x_{2}),(y_{1},y_{2})\in A.$$
Since the multiplication on $A$ is zero, the natural $A$-module actions on $(_0A)^*$ are trivial. Consequently,
$D(xy)=0=x\cdot D(y)+D(x)\cdot y$
for all $x,y\in A$. Moreover, for $x=(x_{1},x_{2}))$ and $y=(y_{1},y_{2})$ in $A$, we have
$$\langle D(x),y\rangle+\langle D(y),x\rangle
=x_{1}y_{2}-x_{2}y_{1}
+y_{1}x_{2}-y_{2}x_{1}
=0.$$
Hence, $D$ is a $0$-cyclic derivation. On the other hand, every inner derivation from $A$ into $(_0A)^{*}$ is zero, since the module actions are trivial. However, $D$ is nonzero. Indeed, for
$e_{1}=(1,0)$ and $ e_{2}=(0,1),$ we have $$\langle D(e_{1}),e_{2}\rangle=1.$$
Therefore, $D$ is a non-inner $0$-cyclic derivation. It follows that $A$ is not $0$-cyclic amenable. Since $0\in\Delta(A)\cup\{0\}$, we conclude that $A$ is not character cyclic amenable.
\end{example}


Recall that the second dual $ A^{**}$ of every Banach algebra $A$ carries two canonical Banach algebra structures, the two Arens products $\Box$ and $\Diamond$, and $A$ embeds isometrically into both $(A^{**},{\Box})$ and $(A^{**},{\Diamond})$ via the canonical injection $k(a)(f)=f(a)k(a)(f)=f(a)$. 

\begin{proposition}\label{prop 2.2}
Suppose that  $A$ and $B$ are Banach algebras and  $\Gamma_1:A\longrightarrow B$ and $ \Gamma_2:B\longrightarrow A$ are continuous homomorphisms such that $\Gamma_1\circ \Gamma_2=I_B$, then the following statements are valid:
\begin{itemize}
\item[(i)] If $A$ is  character cyclic amenable, then so is $B$;
\item[(ii)] If $A^{**}$ is  character cyclic amenable, then so is $B^{**}$.
\end{itemize}
\end{proposition}

\begin{proof}
(i) Suppose that $A$ is  character cyclic amenable and $\psi\in \Delta(B)\cup \{0\}$. Let $D:B\longrightarrow (_\psi B)^*$  be a $\psi$-cyclic derivation. Define $\tilde{D}:A\longrightarrow (_{\psi\circ \Gamma_1}A)^*$ by $\tilde{D}(a)=\Gamma_1^*\circ D\circ \Gamma_1(a)~(a\in A).$ It is easy to check that $\tilde{D}$ is a $\psi\circ \Gamma_1$-cyclic derivation. Therefore, there exists $f\in A^*$ such that 
\[\tilde{D}(a)=a\cdot f-\psi\circ \Gamma_1(a) f~~~(a\in A).\]
Since  $\Gamma_2^*\circ \Gamma_1^*=I_{B^*}$, for every $b\in B$, we have
\begin{align*}
D(b)&=\Gamma_2^*\circ \Gamma_1^*\circ D\circ \Gamma_1\circ \Gamma_2(b)\\
&=\Gamma_2^*\big(\tilde{D}(\Gamma_2(b))\big)\\
&=\Gamma_2^*\big(\Gamma_2(b)\cdot f-\psi\circ \Gamma_1(\Gamma_2(b))f\big)\\
&=b\cdot \Gamma_2^*(f)-\psi(b)\Gamma_2^*(f).
\end{align*}
This means  $D$ is inner,  so $B$ is $\psi$-cyclic amenable. Therefore, $B$ is character cyclic amenable.

(ii) Since $\Gamma_1^{**}\circ \Gamma_2^{**}=I_{B^{**}}$, the proof is similar to (i).
\end{proof}
\begin{corollary}\label{lemma 2.3}
Let $A$ be a Banach algebra such that $A=B\oplus I$ for a closed subalgebra $B$ and closed ideal $I$. Suppose that $A$ is  character cyclic amenable, then so is $B$.
\end{corollary}
\begin{proof}
Suppose that $A$ is  character cyclic amenable. Let  $\Gamma_1:A\longrightarrow B$ be the natural projection and $\Gamma_2:B\longrightarrow A$ be defined by $\Gamma_2(b)=(b,0)~(b\in B)$. Since $\Gamma_1$ and $\Gamma_2$ are continuous homomorphisms, from  Proposition \ref{prop 2.2}, it follows that $B$ is  character cyclic amenable.
\end{proof}


A Banach algebra $A$ is said to be dual if there is a closed submodule $A_*$ of $A^*$ such that $A=(A_*)^*$.
The study of dual Banach algebras provides a natural framework for extending the properties of a Banach algebra $A$ to its second dual $A^{**}$. Given that a dual Banach algebra $A$ is identified with the dual of a closed submodule $A_*\subseteq A^*$ , one may investigate the inheritance of various amenability-type properties between the algebra and its dual space. In particular, when $A$ carries the structure of a dual Banach algebra, the relationship between the character cyclic amenability of $A$ and its second dual becomes of significant interest. The following result demonstrates that the character cyclic amenability of $A^{**}$ effectively descends to the algebra $A$ under the dual structure.

  \begin{theorem}
  Let $A$ be a dual Banach algebra. If $A^{**}$ is  character cyclic amenable, then so is $A$. 
  \end{theorem}
\begin{proof}
Let $\phi\in \Delta(A)\cup\{0\}$ and $D:A\longrightarrow (_\phi A)^*$ be a $\phi$-cyclic derivation. Let $i:A_*\longrightarrow A^*$ be the canonical embedding and $i^*$ be the adjoint of $i$. Obviously, $i$ is norm-continuous, hence $i^*$ is $w^*$-continuous. Also, one can easily show that 
\begin{eqnarray}\label{2.2}
 i^*(FG)=i^*(F)i^*(G)~~~(F, G\in A^{**}).
\end{eqnarray}
Let $\mathfrak{D}:A^{**}\longrightarrow (_{\phi\circ i^*} A^{**})^*$ defined by $\mathfrak{D}(F)=i^{**}\circ D\circ i^*(F)~(F\in A^{**})$. 
By \eqref{2.2} for every $F, G, H\in A^{**}$, we have 
 \begin{align*}
\langle H,\mathfrak{D}(FG)\rangle &=\langle H,i^{**}\circ D\circ i^*(FG)\rangle\\
&=\langle i^*(H), D\big( i^*(F)i^*(G)\big)\rangle\\
&=\langle i^*(H),\phi\big(i^*(G)\big)D\big(i^*(F)\big)\\
&\quad +i^*(F)\cdot D\big(i^*(G)\big)\rangle\\
&=\langle H,\phi\big(i^*(G)\big)i^{**}\circ D\circ i^*(F)\\
&\quad +i^{**}\big(i^*(F)\cdot D\big(i^*(G)\big)\big)\rangle\\
&=\langle H,\phi\circ i^*(G)\mathfrak{D}(F)+F\cdot\mathfrak{D}(G)\rangle.
\end{align*}
Also 
\begin{align*}
\langle F,\mathfrak{D}(G)\rangle+\langle G,\mathfrak{D}(F)\rangle &=\langle F,i^{**}\circ D\circ i^*(G)\rangle+\langle G,i^{**}\circ D\circ i^*(F)\rangle\\
&=\langle i^*(F), D\big( i^*(G)\big)\rangle+\langle i^*(G), D\big( i^*(F)\big)\rangle=0.
\end{align*}

Thus $\mathfrak{D}$ is a $\phi\circ i^*$-cyclic derivation. Now,  character cyclic amenability of $A^{**}$, implies  that there exists $\mathfrak{F} \in A^{***}$ such that $\mathfrak{D}(F)=F\cdot \mathfrak{F}-\phi\circ i^*(F)\mathfrak{F}~(F\in A^{**}).$ The canonical mapping $p:A\longrightarrow A^{**}$ is an $A$-bimodule morphism, and hence is $p^*:A^{***}\longrightarrow A^*$. Set $f=p^*(\mathfrak{F})$. Then for every $a,b\in A$, we have 
\begin{align*}
\langle b,D(a)\rangle&= \langle i^*(\widehat{b}),D\big(i^*(\widehat{a})\big)\rangle\\
&=\langle \widehat{b},i^{**}\circ D\circ i^*(\widehat{a})\rangle\\
&=\langle \widehat{b},\widehat{a}\cdot \mathfrak{F}-\phi\circ i^*(\widehat{a})\mathfrak{F}\rangle\\
&=\langle p(b), \widehat{a}\cdot \mathfrak{F}-\phi\circ i^*(\widehat{a})\mathfrak{F}\rangle\\
&=\langle b,p^*\big(\widehat{a}\cdot \mathfrak{F}-\phi\circ i^*(\widehat{a})\mathfrak{F}\rangle\big)\rangle\\
&=\langle b, a\cdot f-\phi(a)f\rangle.
\end{align*}
This means that $D$ is inner.  So $A$ is  $\phi$-cyclic amenable and  this completes the proof.
\end{proof}

    \section{Character cyclic amenability of certain Banach algebras}
     We begin this section with the following :
    
   For a Banach algebra $A$, let $A^{op}$ be a Banach algebra, whose underlying Banach space is $A$, but the product is $ \bullet $, where $a\bullet b=ba ~(a,b \in A)$.
   
   \begin{proposition}
   Let $A$ be a Banach algebra and $\phi\in \Delta(A)\cup \{0\}$. Then $A$ is character cyclic amenable if and only if $A^{op}$ is  character cyclic amenable.
   \end{proposition}
    
    \begin{proof}
Without loss of generality, assume  $A$ is  $\phi$-cyclic amenable. Let $D:A^{op}\longrightarrow (_{\phi}A^{op})^*$ be a continuous $\phi$-cyclic derivation. Define $\tilde{D}:A\longrightarrow (_\phi A)^*$ by $\tilde{D}(a)=i^*\circ D\circ i(a),$ where $i:A\longrightarrow A^{op}$ is a continuous anti-isomorphism. One can easily check that $\tilde{D}$ is a continuous $\phi$-cyclic derivation. Since $A$ is $\phi$-cyclic amenable, there exists $f\in (_\phi A)^*$ such that $$\tilde{D}(a)=a\cdot f-f\cdot a=a\cdot f-\phi(a)f.$$ Now, for every $a'\in A$, we have: 
    \begin{align*}
\langle a',D(a)\rangle&=\langle a', i^*\circ D\circ i(a)\rangle =\langle a',a\cdot f-\phi(a)f\rangle \\
&=\langle a', f\bullet a-\phi(a)f\rangle=\langle a', f\bullet a-a\bullet f\rangle\\
& =\langle a', a\bullet(-f)-(-f)\bullet a\rangle.
\end{align*}
It follows that  $D=\text{ad}_{-f}$. Therefore, $A^{op}$ is $\phi$-cyclic amenable. Since $(A^{op})^{op}=A$, the converse follows similarly. 
    \end{proof}
    
     It is well known that $A\widehat{\otimes}B$, the projective tensor product of $A$ and $B$ is a Banach algebra with respect to the canonical multiplication defined by 
    $$(a_1\otimes b_1)(a_2\otimes b_2)=(a_1a_2\otimes b_1b_2)~~~(a_1, a_2\in A, b_1, b_2\in B).$$ 
For $f\in A^*$ and $g\in B^*$, let $f\otimes g$ denote the element of $(A\widehat{\otimes}B)^*$ satisfying $(f\otimes g)(a\otimes b)=f(a)g(b)$ for all $a\in A$ and $b\in B$. Then, with this notion, 
 $$\Delta(A\widehat{\otimes}B)=\{\phi\otimes \psi:\phi\in \Delta(A), \psi\in \Delta(B)\}.$$
 
 Next, we examine the behaviour  of character cyclic amenability under tensor products.
    \begin{theorem}\label{Thm 4.2}
    Let $A$ and $B$ be Banach algebras. If $A\widehat{\otimes}B$ is character cyclic amenable, then $A$ and $B$ are character cyclic amenable.
    \end{theorem}
   
   \begin{proof}
   let $\phi \in\Delta(A)\cup \{0\}$,  $\psi\in \Delta(B)$ and $b_0\in B$  be such that $\psi(b_0)=1$. Assume that $d:A\longrightarrow (_\phi A)^*$ is a $\phi$-cyclic derivation. Define $D:A\widehat{\otimes}B \longrightarrow (_{(\phi\otimes\psi)}A\widehat{\otimes}B)^*$ by 
   $$\langle a'\otimes b',D(a\otimes b)\rangle=\langle a',d(a)\rangle \psi(b'b) ~(a,a'\in A, b, b'\in B).$$ For every $a_1, a_2,a'\in A$ and $b_1, b_2, b'\in B$, one can verify that: 
   \begin{align*}
\langle a'\otimes b',& a_1\otimes b_1\cdot D(a_2\otimes b_2)\rangle +\langle a'\otimes b', D(a_1\otimes b_1)\cdot a_2\otimes b_2\rangle \\
&=\langle a'a_1\otimes b'b_1, D(a_2\otimes b_2)\rangle+\langle a'\otimes b', \phi(a_2)\psi(b_2)D(a_1\otimes b_1)\rangle\\
&=\langle a'a_1, d(a_2)\rangle\psi(b'b_1b_2)+\langle a', d(a_1)\rangle\phi(a_2)\psi(b'b_1b_2)\\
& =\langle a', a_1\cdot d(a_2)+\phi(a_2)d(a_1)\rangle\psi(b'b_1b_2)\\
&=\langle a',d(a_1a_2)\rangle \psi(b'b_1b_2)=\langle a'\otimes b', D(a_1\otimes b_1a_2\otimes _2).
\end{align*}
Also
$$\langle a'\otimes b', D(a\otimes b)\rangle+\langle a\otimes b, D(a'\otimes b')\rangle=\langle a',d(a)\rangle +\langle a,d(a')\rangle \psi(bb')=0.$$
Thus $D$ is a $\phi\otimes\psi$-cyclic derivation. Now character cyclic amenability of $A\widehat{\otimes}B$ implies that, there exists $F\in (_{(\phi\otimes\psi)}A\widehat{\otimes}B)^*$ such that $$D(a\otimes b)=a\otimes b\cdot F-(\phi\otimes\psi)(a\otimes b)F ~(a\in A, b\in B).$$ Define $f:A\longrightarrow \Bbb C$ by $f(a)=F(a\otimes b_0^2)~(a\in A).$ Then for every $a, a'\in A$, we have:
 \begin{align*}
\langle a',d(a)\rangle &=\langle a', d(a)\rangle \psi(b_0^2)=\langle a'\otimes b_0, D(a\otimes b_0)\rangle\\
&=\langle a'\otimes b_0, a\otimes b_0\cdot F-F\cdot a\otimes b_0\rangle\\
&=\langle a'a\otimes b_0^2,F\rangle-\langle a'\otimes b_0,\phi(a)\psi(b_0)F\rangle\\
&=\langle a'a,f\rangle-\phi(a)\langle a',f\rangle=\langle a',a\cdot f-\phi(a)f\rangle.
\end{align*}
So $d(a)=a\cdot f-\phi(a)f$, which implies  $A$ is character cyclic amenable. The same conclusion holds for  character cyclic amenability of $B$.
\end{proof}

 The converse of Theorem \ref{Thm 4.2} does not hold in general.
\begin{example}
 Let $A=\mathbb C$ be endowed with the zero multiplication. Since $\Delta(A)=\varnothing$, it suffices to consider $0$-cyclic derivations. Let
$D:A\longrightarrow (_0A)^*$
be a $0$-cyclic derivation. Since $A$ is one-dimensional, the cyclicity condition gives
$2\langle D(1),1\rangle=0,$ and hence $D=0$. Thus $A$ is character cyclic amenable.
Now let $B=\mathbb C^2$ with coordinatewise multiplication. Since $B$ is a unital Banach algebra, Theorem \ref{the 2.3} implies that $B$ is character cyclic amenable.
However, $A\widehat{\otimes}B$ has zero multiplication. Indeed, for $a,a'\in A$ and $b,b'\in B$,
$$(a\otimes b)(a'\otimes b')=aa'\otimes bb'=0.$$
Moreover, $A\widehat{\otimes}B$ is two-dimensional and is therefore isomorphic, as a Banach algebra, to $\mathbb C^2$ equipped with the zero multiplication. By Example \ref{exm 3.5}, this algebra is not character cyclic amenable. Consequently, $A\widehat{\otimes}B$ is not character cyclic amenable, although both $A$ and $B$ are character cyclic amenable. Hence, the converse of Theorem \ref{Thm 4.2} is false in general.
\end{example}
 
 \begin{proposition}\label{prop3.4}
 Let $A$ and $B$ be Banach algebras, $h:A\longrightarrow B$ be a continuous homomorphism with dense range and  $\psi\in \Delta(B)\cup\{0\}$. If $(_{\psi}B)^*\circ h=(_{\psi\circ h}A)^*$, then  $\psi\circ h$-cyclic amenability of $A$ implies the $\psi$-cyclic amenability of $B$.
 \end{proposition}
 
 \begin{proof}
 Let $\psi\in \Delta(B)\cup\{0\}$ and $D_B:B\longrightarrow (_{\psi}B)^*$ be a $\psi$-cyclic derivation. Define $D_A:A\longrightarrow (_{\psi\circ h}A)^*$ by $D_A(a)=h^*\circ D_B\circ h(a)~(a\in A)$. One can easily check that $D_A$ is a $\psi\circ h$-cyclic derivation. Thus, there exists $g\in (_{\psi}B)^*$ such that 
 $$D_A(a)=a\cdot (g\circ h)-\psi\circ h(a)(g\circ h)~(a\in A).$$
 Hence, for every $a, a'\in A$ we have:
 \begin{align*}
\langle h(a'),D_B(h(a))&=\langle a',D_A\rangle\\
&=\langle a',a\cdot (g\circ h)-\psi\circ(a)(g\circ h)\rangle\\
&=\langle h(a'),h(a)\cdot g-\psi\circ h(a) g \rangle.
\end{align*}
Since  $h(A)$  is dense in  $B$, this relation extends to all $b\in B$: 
$$D_B(b)=b\cdot g-\psi(b)g~(g\in (_{\psi}B)^*).$$
 That is, $D_B=\text{ad}_g$, which means that $B$ is $\psi$-cyclic amenable.
 \end{proof}
 

For a Banach algebra $A$,  $\phi\in \Delta(A)$  and a closed two-sided ideal $I$ of $A$ such
that $I\subseteq ker(\phi)$, we can extend $\phi$  to $\hat{\phi}\in \Delta(A/I)$  defined by $\hat{\phi}(a + I) = \phi(a)$ for every $a\in A$. Thus, in the next result we study the relation between $\phi$-cyclic amenability  of A and  $\hat{\phi}$-cyclic amenability of $A/I$.
\begin{corollary}
Let $A$ be a Banach algebra,  $\phi\in \Delta(A)$ and $I$ be a closed ideal of $A$ such that  $I\subseteq ker(\phi)$. If $A$ is $\phi$-cyclic amenable, then $A/I$ is $\hat{\phi}$-cyclic amenable.
\end{corollary}

\begin{corollary}
Let $A$ be a Banach algebra and $B$ be a dense subalgebra of $A$. If $B$ is character cyclic amenable, then so is $A$.
\end{corollary}

\begin{proof}
Let $i:B\longrightarrow A$ be the inclusion map and $\phi\in \Delta(A)\cup\{0\}$. Then $i$ is the canonical continuous homomorphism with dense range. For every $g\in B^*$, there exists $f\in A^*$ such that $f|_B=g.$ Thus $f\circ i=g$, and so $(_{\phi}A)^*\circ i=(_{\phi\circ i}B)^*$. Now the result follows from Proposition $\ref{prop3.4}$.
\end{proof}


\begin{theorem}\label{theo 4.7}
Let $A$ be a Banach algebra,  $\phi\in \Delta(A)$ and $I$ be a closed ideal of $A$ such that  $I\subseteq ker(\phi)$.   If  $I$ is $\phi_I$-cyclic amenable and $A/I$ is $\hat{\phi}$-cyclic amenable, then $A$ is $\phi$-cyclic amenable.
\end{theorem}

\begin{proof}
Suppose that  $I$ is $\phi_I$-cyclic amenable and $A/I$ is $\hat{\phi}$-cyclic amenable. 
 Let $D:A\longrightarrow (_\phi A)^* $ be a $\phi$-cyclic derivation. Define $D_I:I \longrightarrow (_\phi I)^*$ by $D_I(i)=D(i)~(i\in I)$. By hypothesis, there exists $g_1\in  (_\phi I)^*$ such that $D_I(i)=i\cdot g_1-\phi(i)g_1~(i\in I)$. Since $I$ is a closed subspace of the Banach space $A$, and $g_1 : I \to \mathbb{C}$ is a bounded linear functional, the Hahn--Banach theorem ensures the existence of a bounded linear extension $\tilde{g} : A \to \mathbb{C}$ satisfying $\tilde{g}|_I = g_1$. 
 Now define $D':A\longrightarrow (_\phi A)^*$ by 
 $$ D'(a)=D(a)-(a\cdot \tilde{g}-\tilde{g}\cdot a)~(a\in A). $$
  One can easily check that  $D'$ is a $\phi$-cyclic derivation and $D'|_I=0$. Hence, the map 
  $$\bar{D}:A/I\longrightarrow (_{\hat{\phi}}A/I)^*, ~~~a+I\longrightarrow D'(a),$$ is well-defined. Since $\bar{D}(a+I)\in I^{\perp}=\{f\in A^*: f|_I=0\}$ for every $a\in A$, we have: 
 $$ \langle a_1+I, \bar{D}(a_2+I)\rangle+\langle a_2+I,\bar{D}(a_1+I)\rangle=\langle a_1, D'(a_2)\rangle+\langle a_2, D'(a_1)\rangle=0 $$
for all $a_1, a_2\in A$. It is clear that $\bar{D}$ is $\hat{\phi}$-cyclic derivation. Thus, there exists $h\in (_{\hat{\phi}}A/I)^*$ such that $$\bar{D}(a+I)=(a+I)\cdot h-\hat{\phi}(a+I)h~(a\in A).$$ Define $\tilde{h}$ by  $\tilde{h}(a)=h(q(a))~(a\in A)$, where $q:A\longrightarrow A/I$ be the quotient map. Hence, for every $a\in A$, we have:
$$ D'(a)=\bar{D}(a+I)=(a+I)\cdot h-\hat{\phi}(a+I)h=a\cdot \tilde{h}-\phi(a)\tilde{h}.$$ 
Let $f=\tilde{h}+\tilde{g}$. Then 
$$D(a)=D'(a)+(a\cdot \tilde{g}-\tilde{g}\cdot a)=a\cdot f-\phi(a)f$$
for all $a\in A$. Therefore, $D$ is inner, and   $A$ is $\phi$-cyclic amenable.

\end{proof}


Let $A$ and $B$  be  Banach algebras; it is well known that $A\oplus_1 B$, the $l^1$-direct sum of $A$ and $B$, is a Banach algebra with respect to the canonical multiplication defined by $ (a,b)(c,d):= (ac,bd).$ Since $(A\oplus B)^*=(0\oplus B)^{\perp}$\.{+} $(A\oplus 0)^{\perp}$, where \.{+}  denotes the $l^{\infty}$-direct sum, and $(0\oplus B)^{\perp}$ (resp. $(A\oplus 0)^{\perp}$) is isometrically isomorphic to $A^*$ (resp. $B^*$), for convenience, we write:       $(A\oplus_1 B)^*=A^*$\.+ $B^*.$   Moreover, 
\[\Delta(A\oplus_1 B)=(\Delta(A)\times \{0\})\cup (\{0\}\times \Delta(B)).\]
  \begin{theorem} Let $A$ and $B$  be  Banach algebras. Then  $A\oplus_1 B$ is character cyclic amenable if and only if  $A$ and $B$ are character cyclic amenable.
   \end{theorem}
  \begin{proof}
  Suppose that  $A\oplus_1 B$ is character cyclic amenable. Let $\phi\in \Delta(A)\cup\{0\}$  and $D_A:A\longrightarrow (_{\phi}A)^*$ be a  $\phi$-cyclic derivation. Define $D:A\oplus_1 B\longrightarrow (_{(\phi,0)}A\oplus_1 B)^*$ by $$\langle (a',b'),D(a,b)\rangle=\langle a', D_A(a)\rangle~(a\in A, b\in B).$$
 It is straightforward to verify that  $D$ is a $(\phi,0)$-cyclic derivation. By the character cyclic amenability of $A\oplus_1 B$,  there exists   $(f,g)\in (_{(\phi,0)}A\oplus_1 B)^*$ such that 
 $$D(a,b)=(a,b)\cdot (f,g)-(\phi,0)(a,b)(f,g)~(a\in A, b\in B).$$ Hence, for every $a, a'\in A$, we have 
\begin{align*}
\langle a',D_A(a)\rangle&=\langle (a',0),D(a,0)\rangle\\
&=\langle (a',0),(a,0)\cdot (f,g)-(\phi,0)(a,0)(f,g)\rangle\\
&=\langle a', a\cdot f-\phi(a)f\rangle.
\end{align*}  
  This implies that  $D_A$ is inner and therefore $A$ is $\phi$-cyclic amenable. Similarly, we may show that $B$  is character cyclic amenable.
   
 Conversely, suppose that $A$ and $B$ are character cyclic amenable. We prove that $A\oplus_1B$ is character cyclic amenable. First, let $\Phi=(\phi,0)~( \phi\in\Delta(A))$. Consider the closed ideal $I={0}\oplus B$.
Clearly, $I\subseteq\ker\Phi$, since $\Phi(0,b)=\phi(0)=0.$  Since $B$ is character cyclic amenable, $B$ is $0$-cyclic amenable. Hence $I$ is $\Phi_I$-cyclic amenable. $(A\oplus_1B)/({0}\oplus B)\cong A$
as Banach algebras. Under this isomorphism, the character induced by $\Phi$ on $(A\oplus_1B)/I$ corresponds precisely to $\phi$. Since $A$ is character cyclic amenable, $(A\oplus_1B)/I$ is $\widehat{\Phi}$-cyclic amenable.
Therefore, by Theorem \ref{theo 4.7}, $A\oplus_1B$ is $(\phi,0)$-cyclic amenable. Similarly, let $\Phi=(0,\psi),~( \psi\in\Delta(B)).$
Taking $I=A\oplus{0},$
we have  $I\subseteq\ker\Phi,$ and $(A\oplus_1B)/I\cong B.$
Since $A$ is $0$-cyclic amenable and $B$ is $\psi$-cyclic amenable, Theorem \ref{theo 4.7} implies that $A\oplus_1B$ is $(0,\psi)$-cyclic amenable. Therefore, $A\oplus_1B$ is character cyclic amenable.
 \end{proof}


Let $A$ and $B$ be Banach algebras with $\Delta(B)\neq 0$ and let $\theta\in\Delta(B)$. Then the direct product $A\times B$ equipped with the algebra multiplication 
$$(a,b)\cdot (a',b')=(aa'+\theta(b')a+\theta(b)a',bb'),\hspace{.5cm}(a,a'\in A, b,b'\in B),$$ 
and with the norm $\|(a,b)\|=\|a\|+\|b\|$ is a Banach algebra which is called the $\theta$-Lau product of $A$ and $B$,  denoted by $A\times_{\theta}B$. This type of product was introduced by Lau \cite{1110} for a certain class of Banach algebras and was extended by Sangani Monfared \cite{103} to the general case. For example, the unitization $A^\sharp=A\times_{\theta}\Bbb C$ of $A$ can be regarded as the $\theta$-Lau product of $A$ and $\Bbb C$, where $\theta:\Bbb C\longrightarrow \Bbb C$ is the identity map. 
  
It is easy to verify that the dual space $(A\times_{\theta}B)^*$ can be identified with $A^* \oplus_\infty B^*$ under the canonical pairing 
$\langle(f,g),(a,b)\rangle=f(a)+g(b)~(a\in A, b\in B, f\in A^*, g\in B^*).$ Also, for $F=(f,g)\in (A\times_{\theta}B)^* = A^* \oplus_\infty B^*,$ the left and right dual module actions are defined as follows:
$$ (a,b)\cdot(f,g) = \big(af + \theta(b)f, \; f(a)\theta + bg\big) $$
and
$$(f,g)\cdot(a,b) = \big(fa + \theta(b)f, \; f(a)\theta + gb\big),$$
for all $a\in A$ and $b\in B$.

    
  To prove our next result, we need to quote the following proposition from \cite{103}.
 \begin{proposition}\label{prop 3.1}
  Let $A$ and $B$ be two Banach algebras with $\Delta(B)\neq 0$. Let $\theta\in\Delta(B), ~E=\{(\varphi,\theta):\varphi\in \Delta(A)\}$ and $F=\{(0,\psi):\psi\in \Delta(B)\}.$
  Set $E=\emptyset$ if $\Delta(A)=\emptyset$. Then $\Delta(A\times_{\theta} B)=E\cup F.$
  \end{proposition}
  

\begin{proposition}\label{prop 4.10}
   Let $A$ and $B$ be Banach algebras, let $\theta\in\Delta(B)$, let $\phi\in\Delta(A)$, and put
$C=A\times_\theta B,~ \Phi=(\phi,\theta).$
Let $D:C\longrightarrow(_{\Phi} C)^*$ be a continuous $\Phi$-cyclic derivation. Then there exist continuous maps $d:A\longrightarrow A^*$ and $ \delta:B\longrightarrow B^*$ such that
$$D(a,b)=(d(a),\delta(b))~(a\in A,\ b\in B),$$
   where $d$ is a $\phi$-cyclic derivation and $\delta$ is a $\theta$-cyclic derivation. Moreover, $D$ is inner if and only if there exist $f\in A^*$ and $g\in B^*$ such that
$d(a)=af-\phi(a)f,~a\in A$,~ $\delta(b)=bg-\theta(b)g,~ b\in B,$ and
$f(a)\theta=\phi(a)g\qquad(a\in A).$
\end{proposition}

\begin{proof}
 Since $C^*=A^*\oplus_\infty B^*,$
write $$D(a,0)=(d(a),S(a)),~D(0,b)=(T(b),\delta(b)),$$
where
$d:A\to A^*,~ S:A\to B^*,~
T:B\to A^*,~ \delta:B\to B^*$
are continuous linear maps. Thus $$D(a,b)=(d(a)+T(b),S(a)+\delta(b)).$$
  Since $(a,0)(a',0)=(aa',0),$ the derivation property implies that $d$ is a $\phi$-derivation. Similarly, since
$(0,b)(0,b')=(0,bb'),$ the map $\delta$ is a $\theta$-derivation.
We next consider the mixed product. Since $(a,0)(0,b)=(\theta(b)a,0),$ the derivation identity gives
$$D(\theta(b)a,0)=(a,0)\cdot D(0,b)+D(a,0)\cdot(0,b).$$ Using the dual module actions, the first component of the right-hand side is
$\theta(b)d(a)+\phi(a)T(b).$
The first component of the left-hand side is $\theta(b)d(a).$ Consequently,
$\phi(a)T(b)=0~(a\in A,\ b\in B).$
Since $\phi\neq0$, it follows that $T=0.$ On the other hand, the $\Phi$-cyclicity of $D$, applied to $(a,0)$ and $(0,b)$, gives $$\langle D(a,0),(0,b)\rangle+\langle D(0,b),(a,0)\rangle=0.$$
Hence $S(a)(b)+T(b)(a)=0.$ Since $T=0$, we obtain $S=0.$ Therefore $D(a,b)=(d(a),\delta(b)).$ The cyclicity of $D$, restricted to $A\times{0}$ and ${0}\times B$, now yields
$$\langle d(a),a'\rangle+\langle d(a'),a\rangle=0$$ and 
$$\langle\delta(b),b'\rangle+\langle\delta(b'),b\rangle=0. $$
 Thus $d$ is $\phi$-cyclic and $\delta$ is $\theta$-cyclic.   Therefore, $D=\operatorname{ad}_{(f,g)}$ if and only if 
\begin{align*}
(d(a),\delta(b))&=(a,b)\cdot (f,g)-(\phi,\theta)(a,b)(f,g)\\
&=\left(af-\phi(a)f,f(a)\theta-\phi(a)g+bg-\theta(b)g\right),
\end{align*}
for all $a\in A$ and $b\in B$. Thus, $D$ is inner if and only if 
$d(a)=af-\phi(a)f,~\delta(b)=bg-\theta(b)g,$
and $f(a)\theta-\phi(a)g=0$ for all $a\in A$. This proves the result.
\end{proof}


\begin{theorem}\label{thm 4.11}
 Let $A$ and $B$ be Banach algebras, let $\theta\in\Delta(B)$, and suppose that $A$ and $B$ are character cyclic amenable. Then the following assertions hold.
 \begin{itemize}
 \item[(i)] If $A$ has a bounded approximate identity, then  $A\times_\theta B $ is character cyclic amenable.
\item[(ii)] More generally, suppose that, for every $\phi\in\Delta(A), \overline{\operatorname{span}}
\{xy+yx:x,y\in\ker\phi\}
=\ker\phi.$ Then $A\times_\theta B$ is character cyclic amenable.
 \end{itemize}
\end{theorem}
\begin{proof}
 Put $C=A\times_\theta B.$ We first prove that $C$ is $0$-cyclic amenable. Let
$I={0}\times B.$ Then $I$ is a closed ideal of $C$, $I\subseteq\ker 0$, and $I$ is isomorphic to $B$. Moreover,
$C/I\cong A.$ Since $A$ and $B$ are character cyclic amenable, both $A$ and $B$ are $0$-cyclic amenable. Hence, by Theorem \ref{theo 4.7}, $C$ is $0$-cyclic amenable. Next, let $\psi\in\Delta(B)$ and put $\Phi=(0,\psi)\in\Delta(C).$ Consider the closed ideal $J=A\times{0}.$ Clearly, $J\subseteq\ker\Phi,~J\cong A,~C/J\cong B.$
The character induced by $\Phi$ on $C/J$ is precisely $\psi$. Since $A$ and $B$ are character cyclic amenable, Theorem \ref{theo 4.7} implies that $C$ is $\Phi$-cyclic amenable. 
 
 It remains to consider $\Phi=(\phi,\theta),~ \phi\in\Delta(A).$
Let $D:C\longrightarrow(_\Phi C)^*$ be a continuous $\Phi$-cyclic derivation. By Proposition 4.11, there exist continuous linear maps $d:A\longrightarrow A^*,~\delta:B\longrightarrow B^*$ such that $D(a,b)=(d(a),\delta(b)),$ where $d$ is a $\phi$-cyclic derivation and $\delta$ is a $\theta$-cyclic derivation.
We first show that $\delta=0.$ For $a\in A$ and $b\in B$, we have $(0,b)(a,0)=(\theta(b)a,0).$
Hence, by the derivation identity,
\begin{align*}
D(\theta(b)a,0)&=(0,b)\cdot D(a,0)+D(0,b)\cdot(a,0)\\
&=(0,b)\cdot (d(a),0)+\phi(a)(0,\delta(b))\\
&=(\theta(b)d(a),\phi(a)\delta(B)).
\end{align*}
 \begin{equation}
\phi(a)\delta(b)=0\qquad(a\in A,\ b\in B).
\tag{4.17}
\end{equation}
Since $\phi\neq0$, there exists $a_0\in A$ such that $\phi(a_0)\neq0.$ Taking $a=a_0$, we obtain
\begin{equation}
\delta(b)=0\qquad(b\in B).
\tag{4.18}
\end{equation}
Thus $\delta=0,$ and therefore $D(a,b)=(d(a),0)$.
Since $A$ is character cyclic amenable and $d$ is a $\phi$-cyclic derivation, $d$ is inner. Hence there exists $f\in A^*$ such that $d(a)=af-\phi(a)f~(a\in A).$ Since $d$ is cyclic, for $a,b\in A$ we have
$$\langle d(a),b\rangle+\langle d(b),a\rangle=0.$$ This gives
\begin{equation}
f(ab)-\phi(a)f(b)+f(ba)-\phi(b)f(a)=0,
\tag{4.19}
\end{equation}
and hence
\begin{equation}
f(ab+ba)=\phi(a)f(b)+\phi(b)f(a)\qquad(a,b\in A).
\tag{4.20}
\end{equation}
We now prove that, under either hypothesis (i) or (ii), $f=c\phi$ for some $c\in\mathbb C$.
Case (i). Suppose that $A$ has a bounded approximate identity $(e_\alpha)$. Taking $b=e_\alpha$ , we obtain
\begin{equation}
f(ae_\alpha+e_\alpha a)=\phi(a)f(e_\alpha)+\phi(e_\alpha)f(a).
\tag{4.22}
\end{equation} 
Since $(e_\alpha)$ is a two-sided approximate identity, $ae_\alpha+e_\alpha a\longrightarrow2a.$ Also,
$\phi(e_\alpha)\longrightarrow1.$ Because $(e_\alpha)$ is bounded and $f\in A^*$, the scalar net
$(f(e_\alpha))$ is bounded. Hence it has a convergent subnet, say
$f(e_{\alpha_\beta})\longrightarrow c$ for some $c\in\mathbb C$. Passing to this subnet, we obtain
$2f(a)=c\phi(a)+f(a),$ and therefore $f(a)=c\phi(a)~(a\in A).$
Case (ii). Suppose that, for every $\phi\in\Delta(A)$,
\begin{equation}
\overline{\operatorname{span}}\{xy+yx:x,y\in\ker\phi\}=\ker\phi.
\tag{4.23}
\end{equation}
Choose $a_0\in A$ such that $\phi(a_0)=1.$ Put $c=f(a_0)$ and define $h=f-c\phi.$ Then
\begin{equation}
h(a_0)=0.
\tag{4.24}
\end{equation}
Since both $f$ and $\phi$ satisfy $f(ab+ba)=\phi(a)f(b)+\phi(b)f(a)~(a,b\in A),$ we obtain
\begin{equation}
h(ab+ba)=\phi(a)h(b)+\phi(b)h(a)\qquad(a,b\in A).
\tag{4.25}
\end{equation}
If $x,y\in\ker\phi$, then $h(xy+yx)=0.$ Therefore $h$ vanishes on
$\overline{\operatorname{span}}\{xy+yx:x,y\in\ker\phi\}.$
So,
\begin{equation}
h|_{\ker\phi}=0.
\tag{4.26}
\end{equation}
For arbitrary $a\in A$,
$$\phi(a-\phi(a)a_0)
=\phi(a)-\phi(a)\phi(a_0)=0.$$
Thus $a-\phi(a)a_0\in\ker\phi,$  and  we obtain
$$h(a)=h(a-\phi(a)a_0)+\phi(a)h(a_0)=0.$$
Hence $h=0$, and consequently $f=c\phi.$ 
We now set $g=c\theta\in B^*.$ Then
$$bg-\theta(b)g=bc\theta-\theta(b)c\theta=0\qquad(b\in B),$$
and, since $f=c\phi$,
$$f(a)\theta-\phi(a)g=c\phi(a)\theta-\phi(a)c\theta=0\qquad(a\in A).$$
Together with $d(a)=af-\phi(a)f$
and $\delta(b)=0=bg-\theta(b)g,$
all the conditions in Proposition \ref{prop 4.10} are satisfied. Hence $D$ is inner.
Therefore every continuous $\Phi$-cyclic derivation $D:C\longrightarrow(_\Phi C)^*$ is inner for every
$\Phi\in\Delta(C)\cup\{0\}.$ Consequently, $A\times_\theta B$ is character cyclic amenable.

\end{proof}


 \begin{theorem}
     Let $A$ and $B$ be Banach algebras, let $\theta\in\Delta(B)$, and suppose that $A$ and $B$ are character cyclic amenable. If $\ker\theta$ is $0$-cyclic amenable, then $A\times_\theta B$ is character cyclic amenable.
     \end{theorem}
     \begin{proof}
Put $C=A\times_\theta B$. By similar arguments to those used in the proof of Theorem~\ref{thm 4.11}, one can show that $C$ is $0$-cyclic amenable and $(0,\psi)$-cyclic amenable for every $\psi\in\Delta(B)$.
It remains to consider $\varphi\in\Delta(A)$ and put $\Phi=(\varphi,\theta).$ Define
$Q:C\longrightarrow A^\sharp$ by $ Q(a,b)=(a,\theta(b))~(a\in A, b\in B ).$
It is readily verified that $Q$ is a surjective continuous algebra homomorphism. Moreover, $\ker Q=0\times\ker\theta.$ Thus, $C/(0\times\ker\theta)\cong A^\sharp. $ Since $\Phi(0,b)=\theta(b)=0$ for every $b\in\ker\theta$, we have
$0\times\ker\theta\subseteq\ker\Phi.$ The character induced by $\Phi$ on the quotient
$C/(0\times\ker\theta)\cong A^\sharp$ is given by
$\widetilde{\varphi}(a,\lambda)=\varphi(a)+\lambda.$ Since $A^\sharp$ is unital, it has a bounded left approximate identity. Hence, by Theorem \ref{the 2.3}, $A^\sharp$ is character cyclic amenable and, in particular, it is $\widetilde{\varphi}$-cyclic amenable. On the other hand, $(0\times\ker\theta)$ is $0$-cyclic amenable by the assumption that $\ker\theta$ is $0$-cyclic amenable. Therefore, Theorem \ref{theo 4.7} yields that $C$ is $(\varphi,\theta)$-cyclic amenable.
\end{proof}

  
\begin{proposition}
Let $A$ and $B$ be Banach algebras and $\theta\in \Delta(B)$. If $A\times_{\theta} B$ is character cyclic amenable, then $B$ is character cyclic amenable.
\end{proposition}

\begin{proof}
Let $\pi:A\times_{\theta} B\longrightarrow B$ and $i:B\longrightarrow A\times_{\theta} B$ be the canonical projection and the canonical injection maps, respectively. Since $\pi\circ i=I_B$, it follows from Proposition \ref{prop 2.2} $(i)$  that $B$ is character cyclic amenable.
\end{proof}


The following example shows that character cyclic amenability of $A\times_{\theta} B$ does not imply character cyclic amenability of  $A$.

\begin{example}
Let $A$ be a Banach algebra that is not character cyclic amenable. Let $B=\Bbb C$, and $\theta$ be the identity map on $\Bbb C$. Then $A\times_{\theta} B=A\times_{\theta} \Bbb C=A^\sharp$ is character cyclic amenable by Theorem \ref{the 2.3}.
\end{example}  


\section{character cyclic amenability of module extension Banach algebras}

 Let $A$ be a Banach algebra, $\phi\in \Delta(A)\cup \{0\},$ and $X$ be a Banach right $A$-module. 
 Denote by $A\oplus {_{\phi}X}$ the module extension of $A$ corresponding to $_{\phi}X$;  that is, the $l^1$-direct sum of $A$ and $_{\phi}X$ as a Banach algebra endowed with the product 
 $$(a,x)(b,y)=(ab,\phi(a)y+x\cdot b)\hspace{1cm}(a,b\in A, x,y\in X).$$
Then the dual of  $A\oplus {_{\phi}X}$ can be identified with the $l^{\infty}$-direct sum $A^*\times {_{\phi}X}^*$ via the natural pairing: $$\left\langle (a,x), (f,g)\right\rangle =f(a)+g(x)$$ for all $(f,g)\in (A\oplus {_{\phi}X})^*$ and $(a,x)\in A\oplus {_{\phi}X}.$
 Since  the element $(0,x)$ is nilpotent in  $A\oplus {_{\phi}X}$ for all $x\in X$, it follows that $\Delta(A\oplus {_{\phi}X})=\{(\phi,0):\varphi\in \Delta(A)\}.$
  Furthermore, the $A\oplus {_{\phi}X}$-bimodule actions on $(_{(\phi,0)}(A\oplus {_{\phi}X})^*$  are given by:
$$(a,x)\cdot (f,g)=(a\cdot f+g(x)\phi, a\cdot g)$$
and 
$$ (f,g)\cdot (a,x)=(\phi,0)(a,x)(f,g)=(\phi(a)f,\phi(a)g) $$
where $(a,x)\in A\oplus {_{\phi}X}$ and $(f,g)\in A^*\times {_{\phi}X}^*$. For further details, see \cite{3}, \cite{131} and \cite{110}.
  
  Let $X$ be a Banach algebra. For each $x\in X$, we denote by $\hat{x}\in X^{**}$ the canonical image of $x$ in $X^{**}$, defined by $\hat{x}(f)=f(x),~ f\in X^*$.
   
  The following result characterises  cyclic derivations on $A\oplus {_{\phi}X}$.
  \begin{proposition}\label{prop 3.11}
  Let $A$ be a Banach algebra, $\phi\in \Delta(A)$, and let $X$ be a Banach right $A$-module. Then $D:A\oplus {_{\phi}X}\longrightarrow (_{(\phi,0)}(A\oplus {_{\phi}X}))^*$ is a $(\phi,0)$-cyclic derivation if and only if 
  $$ D(a,x)=(D_A(a)-D_X^*(\hat{x}), D_X(a)),$$
  for all $(a,x)\in A\oplus {_{\phi}X}$ such that 
  \begin{itemize}
  \item[(i)] $D_A:A\longrightarrow (_{\phi}A)^*$ is a $\phi$-cyclic derivation.
  \item[(ii)] $D_X:A\longrightarrow (_{\phi}X)^*$ is a derivation.
  \end{itemize}
  \begin{proof}
  By Theorem 3.1 of \cite{110},  $D:A\oplus {_{\phi}X}\longrightarrow (_{(\phi,0)}(A\oplus {_{\phi}X}))^*$ is a  derivation if and only if 
  $$ D(a,x)=(D_A(a)+T(x), D_X(a)),$$
  for all $(a,x)\in A\oplus {_{\phi}X}$ such that 
  \begin{itemize}
  \item[(i)] $D_A:A\longrightarrow (_{\phi}A)^*$ is a  derivation.
  \item[(ii)] $D_X:A\longrightarrow (_{\phi}X)^*$ is a derivation.
  \item[(iii)] $T:{_{\phi}X}\longrightarrow (_{\phi}A)^*$ is a bounded linear map with $T(x\cdot a)=D_X(a)(x)\phi+\phi(a)T(x),$  and $ \phi(a)T(x)=a\cdot T(x),$ for all $a\in A$ and $x\in X$. 
  \end{itemize}
  Hence, for every $(a,x), (b,y)\in A\oplus {_{\phi}X}$ we have:
  \begin{align*}
\left\langle (a,x),D(b,y)\right\rangle+\left\langle (b,y),D(a,x)\right\rangle &=\left\langle (a,x), (D_A(b)+T(y), D_X(b))\right\rangle \\
&+\left\langle (b,y), (D_A(a)+T(x), D_X(a))\right\rangle \\
&=\left\langle a, D_A(b)\right\rangle+\left\langle a, T(y)\right\rangle+ \left\langle x, D_X(b)\right\rangle\\
&+\left\langle b, D_A(a)\right\rangle+ \left\langle b, T(x)\right\rangle +\left\langle y, D_X(a)\right\rangle.
\end{align*}
  Thus, $D$ is a $(\phi,0)$-cyclic derivation if and only if $D_A$ is a $\phi$-cyclic derivation and $T(x)=-D_X^*(\hat{x})$.
  \end{proof}
  \end{proposition}
  
  By similar arguments to the proof of Proposition \ref{prop 3.11}, the following holds:
    \begin{proposition}\label{prop 3.12}
  Let $A$ be a Banach algebra and let $X$ be a Banach right $A$-module. Then $D:A\oplus {_{0}X}\longrightarrow (_{(0,0)}(A\oplus {_{0}X}))^*$ is a $0$-cyclic derivation if and only if 
  $$ D(a,x)=(D_A(a)-D_X^*(\hat{x}), D_X(a)),$$
  for all $(a,x)\in A\oplus {_{0}X}$ such that 
  \begin{itemize}
  \item[(i)] $D_A:A\longrightarrow (_{0}A)^*$ is a $0$-cyclic derivation.
  \item[(ii)] $D_X:A\longrightarrow (_{0}X)^*$ is a derivation.
  \end{itemize}
  \end{proposition}


\begin{theorem}
  Let $A$ be a Banach algebra, and let $X$ be a Banach right $A$-module. Then the following statements are valid: 
  \begin{itemize}
  \item[(i)] If $A\oplus {_{\phi}X}$ is character cyclic amenable, then  $A$ is character cyclic amenable.
  \item[(ii)] $A\oplus {_{\phi}X}$ is $(\phi,0)$-cyclic amenable if and only if 
  \begin{itemize}
  \item[1.] $A$ is $\phi$-cyclic amenable; and 
  \item[2.] For every derivation $D_X:A\longrightarrow (_{\phi}X)^*$,  there exists $g\in (_{\phi}X)^*$ such that $D_X(a)=a\cdot g-\phi(a)g~(a\in A)$ and $g(x)\phi=-D_X^*(\hat{x})$.
  \end{itemize}
  \item[(iii)] $A\oplus {_{0}X}$ is $(0,0)$-cyclic amenable if and only if  $A$ is $0$-cyclic amenable and for every derivation $D_X:A\longrightarrow (_{0}X)^*$,  there is $g\in (_{0}X)^*$ such that $D_X(a)=a\cdot g~(a\in A)$.
  \end{itemize}
  \end{theorem}

\begin{proof}
(i) Suppose that  $A\oplus {_{\phi}X}$ is character cyclic amenable. Let $h_1:A\oplus {_{\phi}X}\longrightarrow A$, be defined by $h_1(a,x)=a~(a\in A, x\in X)$ and $h_2:A\longrightarrow A\oplus {_{\phi}X}$, be defined by $h_2(a)=(a,0)~(a\in A)$. Since $h_1$ and $h_2$ are continuous homomorphisms such that $h_1\circ h_2=I_A$, Proposition 3.3  implies that  $A$ is character cyclic amenable.

(ii) Suppose that conditions $ 1$ and $2$ hold. Let  $D:A\oplus {_{\phi}X}\longrightarrow (_{(\phi,0)}(A\oplus {_{\phi}X}))^*$ be a $(\phi,0)$-cyclic derivation. By Proposition \ref{prop 3.11}, $D$ is of the form
$$D(a,x)=(D_A(a)-D_X^*(\hat{x}), D_X(a)),\qquad (a,x)\in A\oplus {_{\phi}X}.$$
where  $D_A$ is a $\phi$-cyclic derivation and  $D_X$ is  a derivation. By the assumptions, there exist $f\in (_{\phi}A)^*$ and $g\in (_{\phi}X)^*$ such that $D_A(a)=a\cdot f-\phi(a)f~(a\in A)$, $D_X(a)=a\cdot g-\phi(a)g~(a\in A)$ and $g(x)\phi=-D_X^*(\hat{x})~(x\in X)$.  It follows that,
  \begin{align*}
 (a,x)\cdot (f,g)-(f,g)\cdot (a,x)&=(a\cdot f+g(x)\phi, a\cdot g)-(\phi(a)f,\phi(a)g)  \\
&=(a\cdot f-\phi(a)f+g(x)\phi, a\cdot g-\phi(a)g)\\
&=(D_A(a)-D_X^*(\hat{x}),D_X(a))=D(a,x),
\end{align*}
for all $(a,x)\in A\oplus {_{\phi}X}$. Thus, $D$ is an inner derivation, and so $A\oplus{_{\phi}X}$ is $(\phi,0)$-cyclic amenable. 

Conversely, suppose that $A\oplus {_{\phi}}X$ is $(\phi,0)$-cyclic amenable. From (i),  $A$ is $\phi$-cyclic amenable, satisfying condition 1. Let $D_X:A\longrightarrow (_{\phi}X)^*$ be a derivation. Define $D:A\oplus {_{\phi}X}\longrightarrow (_{(\phi,0)}(A\oplus {_{\phi}X}))^*$, by 
$$D(a,x)=(-D_X^*(\hat{x}),D_X(a))~((a,x)\in A\oplus {_{\phi}X}).$$
By Proposition \ref{prop 3.11}, $D$ is a $(\phi,0)$-cyclic derivation. Since  $A\oplus {_{\phi}X}$ is $(\phi,0)$-cyclic amenable, there exist  $f\in (_{\phi}A)^*$ and $g\in (_{\phi}X)^*$ such that 
\begin{align*}
 (-D_X^*(\hat{x}),D_X(a))&=D(a,x)=(a,x)\cdot (f,g)-(f,g)\cdot (a,x)\\
&=(a\cdot f+g(x)\phi, a\cdot g)-(\phi(a)f,\phi(a)g)  \\
&=(a\cdot f-\phi(a)f+g(x)\phi, a\cdot g-\phi(a)g),\\
\end{align*}
for every $(a,x)\in A\oplus {_{\phi}X}$. It follows that $D_X(a)=a\cdot g-\phi(a)g~(a\in A)$.  Furthermore, by setting $a=0$, we conclude that  $g(x)\phi=-D_X^*(\hat{x})$, which confirms that  condition 2 holds. 

(iii) The proof is analogous to the above.
\end{proof}

\end{document}